\documentclass{amsart}

\usepackage[T1]{fontenc}
\usepackage{amssymb, amsmath}
\usepackage{dsfont}
\usepackage{hyperref}
\usepackage{array}
\usepackage{enumerate}

\DeclareMathOperator{\Aut}{Aut}
\DeclareMathOperator{\Fix}{Fix}

\DeclareMathOperator{\ecl}{ecl}

\DeclareMathOperator{\etd}{etd}
\newcommand{\exptransdeg}{exponential transcendence degree}

\newtheorem*{prop}{Proposition}

{\begin{list}{}%
         {\setlength{\leftmargin}{#1}}%
         \item[]%
} {\end{list}}

\newtheorem*{lemma}{Lemma}

\theoremstyle{definition}
\newtheorem*{definition}{Definition}

\newcommand{\CB}{{\mathbb B}}

\newcommand{\DC}{{\mathbb C}}
\newcommand{\DR}{{\mathbb R}}

\newcommand{\DQ}{{\mathbb Q}}
\newcommand{\DZ}{{\mathbb Z}}
\newcommand{\DG}{{\mathbb G}}

\newcommand{\Qab}{{\mathbb Q^{\mathrm{ab}}}}
\newcommand{\QabR}{{\mathbb Q^{\mathrm{rab}}}}
\newcommand{\Qalg}{{\mathbb Q^{\mathrm{alg}}}}
\newcommand{\Ker}{{\text{Ker}}}
\newcommand{\E}{{\text{E}}}

\newcommand{\sub}{\subseteq}

\newcommand{\strong}{\lhd}
\newcommand{\td}{\mathrm{trans.deg.}}
\newcommand{\ldim}{\mathrm{lin.dim.}}

\renewcommand{\phi}{\varphi}

\long\def\symbolfootnote[#1]#2{\begingroup%
\def\thefootnote{\fnsymbol{footnote}}\footnote[#1]{#2}\endgroup}

\def\Ind#1#2{#1\setbox0=\hbox{$#1x$}\kern\wd0\hbox to 0pt{\hss$#1\mid$\hss}
\lower.9\ht0\hbox to 0pt{\hss$#1\smile$\hss}\kern\wd0}

\def\Notind#1#2{#1\setbox0=\hbox{$#1x$}\kern\wd0\hbox to 0pt{\mathchardef
\nn=12854\hss$#1\nn$\kern1.4\wd0\hss}\hbox to
0pt{\hss$#1\mid$\hss}\lower.9\ht0 \hbox to
0pt{\hss$#1\smile$\hss}\kern\wd0}

\begin{document}

\title[Algebraic numbers definable in exponential fields]{The algebraic numbers definable in various exponential fields}
\author{Jonathan Kirby \and Angus Macintyre \and Alf Onshuus }
\maketitle

\section{Introduction}

An exponential field (or E-field) is a field, $F$,
of characteristic 0, together with $\E: F\rightarrow F$ satisfying
\begin{itemize}
\item $\E(0)=1$

\item $\E(x+y)=\E(x)\cdot \E(y)$.
\end{itemize}
Every mathematician knows the classical E-fields $\DR$ and
$\DC$. There are also the $LE$-series (see \cite{DMM}), and the
surreal numbers \cite{G}.

More recently, Zilber has produced beautiful ``complex'' examples \cite{Z}. In $\DC$, the kernel of the
exponential map is $2\pi i \DZ$, an infinite cyclic group. In addition, $\DC$ is algebraically closed,
and its exponential map is surjective. Zilber considered E-fields with these properties, which also satisfy
the conclusion of Schanuel's conjecture (see \ref{strong} below), and which are
\emph{strongly exponentially-algebraically closed}, an analogue of being algebraically closed, but taking into
account the exponentiation (see 3.4 below). In this paper we call such E-fields \emph{Zilber fields}. (Other papers use this name
for slightly larger or smaller classes of exponential fields, but the distinction is not important for our purposes.)
There is an excellent exposition of these E-fields by Marker \cite{Marker06}, and a detailed exposition in \cite{K1}.

The complex exponential field $\DC$ also has the property that for any countable subset $X \subseteq \DC$,
there are only countably many $a \in \DC$ which are exponentially algebraic over $X$. This is the
\emph{countable closure property (CCP)} (see \ref{closure} below, or \cite{K2} for more details of exponential
algebraicity). Zilber proved the dramatic result that there is a unique Zilber field (we call it $\CB$) of
cardinality $2^{\aleph_0}$, which satisfies the countable closure property. He has made the profoundly
explanatory conjecture that $\CB \cong \DC$.

Much is known about the logic of these examples. The real E-field
$\DR$, the $LE$-series field, and the surreal numbers are
elementarily equivalent E-fields (\cite{DMM}, \cite{DE}, and
\cite{M1}). They are model-complete, and decidable if Schanuel's
conjecture is true (\cite{W}, \cite{MW}).

It follows from G\"odel's incompleteness theorem that $\DC$ is
undecidable (see e.g. \cite{T}), and it is not model-complete
(\cite{M2}, \cite{Marker06}). The same undecidability argument
works for Zilber's E-fields, and a different argument shows the
failure of model-completeness \cite{K1}.

In this paper we consider, for each example above, the issue of
which algebraic numbers are pointwise definable. For the real
cases the problem is trivial, since one already knows that in
their pure field theory one can define all real algebraic numbers
\cite{T}. The same question (understanding the pointwise definable
points) for the complex exponential field had already been asked
by Mycielski.

In the ``complex'' cases the notion of real abelian algebraic
number is central (see \ref{QabR} below). The main theorems are:

\noindent {\bf Theorem 1}:  For any E-field with cyclic kernel, in particular $\DC$ or the Zilber fields,
all real abelian algebraic numbers are pointwise definable.
\\
{\bf Theorem 2}: For the Zilber fields, the only pointwise definable
algebraic numbers are the real abelian numbers.


\medskip

The conjecture of Zilber above is one of two main open questions
around the complex exponential field, the other being whether the
real subfield is (setwise) definable. They cannot both have a
positive answer, as can be seen for example from Theorem 2. One
step towards Zilber's conjecture would be to show that Theorem 2
holds for the complex exponential field. One might hope this would
be easier than the full conjecture, but we have not been able to
prove it even assuming Schanuel's Conjecture.

\section{Defining the real abelian numbers}\label{Section1}

\subsection{$\QabR$}\label{QabR}
In this section we consider E-fields $F$ where $\Ker:=\{x\in F :
\E(x)=1\}$ is an infinite cyclic group. Let $\tau$ and $-\tau$ be
the generators.

Note that $\left\{ \E\left(\frac{j\tau}{n}\right) : j=0,\dots,
n-1\right\}$ are distinct $n$th roots of 1, so $F\supset U$, the
group of \emph{all} roots of unity. Thus

\[
F\supset \DQ^{ab}=\DQ(U)=\DQ[U],
\]
the maximal abelian extension of $\DQ$. Let $\Qalg$ be the
field-theoretic algebraic closure of $\DQ$ (as an abstract field).

It is important to note that in no algebraically closed field $F$
of characteristic 0 is there a unique subfield $L\subsetneq F$
with $F=L(i)$. It follows by Artin-Schreier theory (see \cite{Ja})
that there always is at least one such $L$. For if $F$ has
transcendence degree $\kappa$ over $\DQ$, pick a transcendence
basis $B$ over $\DQ$ of cardinality $\kappa$ and let $L$ be a
maximal formally real extension of $\DQ(B)$ in $F$. $L$ will be
real-closed. Indeed, by Artin-Schreier, if $F$ is a finite proper
extension of any subfield $L'$, then $F=L'(i)$ and $L'$ is real
closed. Note too that $L=\Fix(\sigma)$, where $\sigma$ is an
involution of $\Aut(F)$. Conversely, the fixed field of any
involution of $F$ is a field $L$ with $F=L(i)$.

An elaboration of such arguments naturally gives an isomorphism
between conjugacy classes of involutions of $\Aut(F)$ and
isomorphism types of real-closed fields of transcendence degree
$\kappa$ over $\DQ$.

Let us apply these ideas to $L=\Qalg$. Any $K$ with $L=K(i)$ is
isomorphic to the field of real algebraic numbers, so there is
just one conjugacy class of involutions in $\Aut(\Qalg)$. There
are, however, $2^{\aleph_0}$ many involutions in this conjugacy
class. This is because $\DQ\left(\left\{\sqrt{p}\ |\  p \text{
prime} \right\}\right)$ has $2^{\aleph_0}$ different orderings
(you can choose, independently for each $p$, which $\sqrt p$ is
positive), and the corresponding real closures are distinct (but
isomorphic). For example, pick a $\sqrt p$. In some real closures
this will be a square, in others $-\sqrt p$ will be a square.

Finally, note that all restrictions to $\Qab$ of involutions in
$\Aut(\Qalg)$ will be the same involution $\sigma_0$,
characterized by $\sigma_0(x)=x^{-1}$ for all $x\in U$. We call
the elements of $\Fix(\sigma_0)$ the \emph{real abelian numbers},
and write $\QabR$ for $\Fix(\sigma_0)$. We will prove in
\ref{subsection1.7} that every element in $\QabR$ is a rational
combination of special values of the cosine function, which are
totally real, so $\QabR$ is totally real. This implies that it is
included in any maximal formally real subfield of $\Qab$. Now
$\QabR$ has only the one extension in $\Qab$, and that is not
formally real, so $\QabR$ can alternatively be characterized as
the unique maximal formally real subfield of $\Qab$, or as the
intersection of $\Qab$ with the field $\DQ^{\mathrm{tr}}$ of totally real numbers.

\subsection{Defining $\DZ$} We can define $\DZ$ as

\[
\{y : \forall x [E(x)=1 \to E(yx) =1]\},
\]
a $\forall$-definition.

We can define $\DQ$ as
\[
\{y : (\exists z,w \in \Ker)[z=wy]\},
\]
an $\exists$-definition.

In $\DC$, there is also an $\exists$-definition of $\DZ$ given by Laczkovich \cite{L}. He used the idea
\[ x\in \DZ \ \Leftrightarrow\  (x\in \DQ \wedge 2^x\in \DQ) \]
but one has to pay attention to the ambiguity in $2^x$, and, in the
general case, to the \emph{existence} of logarithms. Consider the
formula $\Theta(x)$ defined by
\[
\exists t\left[ \E(t)=2 \wedge \E(xt)\in \mathbb Q \wedge x\in
\DQ\right].
\]

\begin{lemma}
Suppose $F\models (\exists t)[\E(t)=2]$. Then
$F\models \Theta(x)$ if and only if $x\in \mathbb Z$.
\end{lemma}
\begin{proof}
Suppose $F\models \Theta(x)$. Then $x=\frac{m}{n}$ with $m,n\in
\DZ, n>0$. Then $\E(xt)=\E\left(\frac{mt}{n}\right)$ and
$\E(xt)^n=2^m$. But $\E(xt)\in \DQ$, so $\frac{m}{n}\in \DZ$, that
is, $x\in \DZ$.

Conversely, suppose $x\in \DZ$, and $\E(t)=2$. Then $\E(tx)=2^x\in
\DQ$.
\end{proof}
Thus if 2 has a logarithm in $F$, $\DZ$ has a
$\exists$-definition. A similar argument works if any prime number
has a logarithm.

\subsection{Defining $\{\tau, -\tau\}$} This two element set
is defined by

\[
x\in \{\tau, -\tau\} \Leftrightarrow \left(\left(x \in
\Ker\right) \wedge \left(\left(\forall y \in \Ker\right)
\left(\exists n\in \DZ\right)
\left[nx=y\right]\right)\right).
\]

The complexity of this definition is $\forall \exists \forall$ for
a general $F$, but only $\forall \exists$ if some prime has a
logarithm.

\subsection{Sine and cosine} We are not able to distinguish $i$
from $-i$ in the complex exponential case. But we can define
cosine and sine there, and the same definitions make sense in any exponential field in which $-1$ is a square, namely:
$$
\begin{array}{lcl}
\cos(x)=y & \Leftrightarrow & \left(\exists j\right)\left[j^2=-1
\wedge
y=1/2\left(\E\left(jx\right)+\E\left(-jx\right)\right)\right] \\
& \Leftrightarrow & \left(\forall j\right)\left[j^2=-1 \rightarrow
y=1/2\left(\E\left(jx\right)+\E\left(-jx\right)\right)\right]
\end{array}
$$
and
$$
\begin{array}{lcl}
\sin(x)=y & \Leftrightarrow & \left(\exists j\right)\left[j^2=-1
\wedge
y=1/2j\left(\E\left(jx\right)-\E\left(-jx\right)\right)\right] \\
& \Leftrightarrow & \left(\forall j\right)\left[j^2=-1 \rightarrow
y=1/2j\left(\E\left(jx\right)-\E\left(-jx\right)\right)\right].
\end{array}
$$

Thus the graphs of cosine and sine are both $\exists$- and
$\forall$-definable. The standard results of elementary
trigonometry are easily proved (just using that $\E$ is a homomorphism), for example:

\begin{enumerate}[i.]
\item $\cos(-x)=\cos(x)$;

\item $\sin(-x)=-\sin(x)$ ;

\item if $j^2=-1$,   $\{x :\sin(x)=0\}=\frac{1}{2j}\Ker$;

\item if $j^2=-1$, $\{x :\cos(x)=0\} =
\left(\frac{1}{4j}\Ker\smallsetminus
\left(\frac{1}{2j}\Ker\right)\right)$ ;

\item if $j^2=-1$, exactly one of $\sin(\alpha/4j)$ and $\sin(-\alpha/4j)$ is 1 and the other is -1 for any $\alpha\in
\Ker\smallsetminus 2\Ker$.
\end{enumerate}

\subsection{Defining $\pi$} We give a definition which is correct
for the complex exponential, and has an unambiguous meaning for any exponential field with cyclic kernel.

From the definition of $\{\tau, -\tau\}$ we get an unambiguous
definition of $\{\frac{\tau}{2j}, \frac{-\tau}{2j}\}$ for any $j$ such that $j^2=-1$.
Think of this two element set as $\{\pi, -\pi\}$ and define $\pi$
as the unique element $t$ of this set with $\sin(t/2)=1$. The other
element is then $-\pi$.

\subsection{Separating $\pm\sqrt{2}$ (for example)}

$\sqrt 2=2\frac{1}{\sqrt 2}$, and $\cos\left(\frac{\pi}{4}\right) = +\frac{1}{\sqrt
2}$, at least in $\DC$. We \emph{define}
in general $+\sqrt{2}=2\cos\left(\frac{\pi}{4}\right)$.

\subsection{Pointwise definition of elements of $\QabR$}\label{subsection1.7}

Let $\alpha\in \Qab$. Then $\alpha\in \DQ[U]$, so it can be expressed as a finite sum:
\[
\alpha=\sum r_n \E\left(s_n\tau\right),
\]
with $r_n\in \DZ, \ s_n\in \DQ$.

Recall that $\sigma_0$ is the involution in $\Aut(\Qab)$ characterized by $\sigma_0(x) = x^{-1}$ for all $x \in U$. Then if $\alpha\in \QabR$ we have
\[
\alpha=\left(\alpha+\sigma_0\left(\alpha\right)\right)/2= \sum r_n
\left(\frac{\E\left(s_n\tau\right)+\E\left(-s_n\tau\right)}{2}\right) =\sum r_n \cos(2\pi s_n)
\]
which is clearly pointwise definable. This proves Theorem 1.

\section{The other direction: Zilber fields.}

\subsection{Partial exponential fields}

It is convenient to consider subfields of an exponential field
which are not closed under exponentiation. Thus we define a
\emph{partial exponential field} to be a field $F$ (of
characteristic zero) together with a $\DQ$-linear subspace $D(F)$
of $F$ and a map $\E: D(F) \to F$ which satisfies
\begin{itemize}
\item $\E(0)=1$ \item $\E(x+y)=\E(x)\cdot \E(y)$.
\end{itemize}
If $F$ is a partial exponential field then we say it is
\emph{generated} by a subset $X$ if and only if $X \cap D(F)$
spans $D(F)$ and $F$ is generated as a field by $X \cup \E(D(F))$.
In particular, we have the notion of $F$ being
\emph{finitely-generated} if a finite such $X$ exists.

An \emph{embedding} of partial exponential fields $\phi: F \to K$
is a field embedding such that, given any $\alpha,\beta\in F$, if
$E_F(\alpha)=\beta$ then $E_K(\phi(\alpha))=\phi(\beta)$. We will
say that $F$ is a partial exponential subfield of $K$ if it is a
subset and the inclusion map is an embedding of partial
exponential fields. Notice that $\mathbb Q$ with
$D(\mathbb Q)=\{0\}$ is a partial exponential subfield of
every partial exponential field. We call it $\DQ_0$.

For another example, consider the subfield $SK = \Qab(2\pi i)$ of
$\DC$, with $D(SK) = \DQ \cdot 2\pi i$, and the restriction of the
complex exponential map. ($SK$ stands for \emph{standard kernel}.)
Then $SK$ is generated as a partial exponential field by the
single element $2\pi i$ because $\E(D(SK)) = U$, the roots of
unity. Clearly $SK$ is not finitely-generated as a pure field,
because it contains $\Qab$.

\subsection{Strong embeddings and Schanuel's Conjecture}\label{strong}

Suppose $F$ is any exponential field, $F_0$ is a partial
exponential subfield of $F$ and let $Y\subset F$.

We will denote by $\td(Y/F_0)$ the (algebraic) transcendence
degree of the field extension $F_0(Y)/F_0$ and by $\ldim_\DQ(X/Y)$
the (linear) dimension of the $\DQ$-vector space spanned by $X
\cup Y$, quotiented by the subspace spanned by $Y$.

We say that $F_0$ is \emph{strongly embedded} in $F$, and write
$F_0 \strong F$ if and only if for every finite subset $X \sub F$
we have
\[\td(X,\E(X)/F_0) \geqslant \ldim_\DQ(X/D(F_0)).\] For example, $\DR$ is not strongly embedded in
$\DC$, because, taking $X = \{i\}$, we have $\td(i,e^i/\DR) = 0$
and $\ldim_\DQ(i/\DR) = 1$.

\medskip

We will say that a partial $E$-field $F$ satisfies the
\emph{Schanuel Condition} (\emph{SC}) if, whenever $\alpha_1,\dots,
\alpha_n$ in $F$ are $\DQ$-linearly independent, the transcendence
degree of $\DQ(\alpha_1, \dots, \alpha_n,E(\alpha_1),
\dots,E(\alpha_n))$ over $\DQ$ is greater than or equal to $n$.
This is equivalent to saying that, for any finite $X \sub F$,
\[\td(X,\E(X)/\DQ) \geqslant \ldim_\DQ(X/0),\]
so it can be equivalently stated as $\DQ_0 \strong F$ (where, as
mentioned before, $\DQ_0$ is the partial $E$-field $\DQ$ with
trivial exponential domain).

The Schanuel condition implies that any nonzero kernel element is transcendental over $\DQ$,
something which is not always true in exponential fields (see Section \ref{otherfields} for some examples).
If $F$ is an exponential field with cyclic kernel which satisfies SC, then the rules of exponentiation
constrain the behaviour of $E$ so strongly that one can find a
embedding of $SK$ into $F$. This embedding is unique modulo
sending $2\pi i$ to either $\tau$ or $-\tau$, so we will identify
the image of the embedding with $SK$ itself (thus identifying
$\tau$ with $2\pi i$). The Schanuel condition then implies that $SK
\strong F$.

Schanuel's conjecture for $\DC$ is equivalent to the statement that the complex exponential field $\DC$ satisfies the Schanuel condition. It can easily be shown that Schanuel's conjecture is also equivalent to the assertion that $SK \strong \DC$.

\subsection{Exponential algebraic closure and CCP}\label{closure}

Given any exponential field $F$ satisfying the Schanuel condition and any finite $X\subset F$ the
function
\[\delta_F(X):=\td(X,\E(X)/\DQ)-\ldim_\DQ(X/0)\]
is always greater than or equal to 0. Now, it may happen that
$\delta_F(X)>\delta_F(X\cup Y)$ but there can be no infinite
descent, so we can define $\etd_F(X)$, the \emph{\exptransdeg} of $X$ in $F$, to be the minimum of $\delta_F(X_1)$
where $X_1\supset X$. (Both $X_1$ and $X$ are assumed to be finite
subsets of $F$.)

For any finite $X\subset F$ we define the \emph{exponential algebraic
closure of $X$ with respect to $F$}, denoted $\ecl_F(X)$, to
be the set of all elements $c$ such that $\etd_F(\{c\}\cup X)=\etd_F(X)$. For infinite $X$, we define $\ecl_F(X) = \bigcup \{\ecl_F(X_0)\mid X_0 \subseteq X \mbox{, finite}\}$.
 We say that $F$ satisfies the
countable closure property (CCP) if the closure $\ecl_F(X)$ of
every countable subset $X\subset F$ is countable. Notice that
given any $X\subset F$, the exponential algebraic closure of $X$ in $F$ is
an exponential field.
The reader may care to look at \cite{K2} for an approach which
does not rely on the Schanuel condition.

\subsection{Definition of Zilber fields}

Recall that in the introduction we defined Zilber fields as
E-fields which are algebraically closed fields with standard
kernel, surjective exponential map, which also satisfy the
conclusion of Schanuel's conjecture, and which are \emph{strongly
exponentially-algebraically closed} -- a notion which we have not
defined yet. We now give the definition for the sake of
completeness, although we do not use it directly in the paper.

\medskip

Let $F$ be any exponential field, let $K$ be a subfield of $F$,
and let $\alpha_1, \dots, \alpha_n\in F$. Suppose (all other cases
reduce to this) that the $\DQ$-linear dimension of $\{\alpha_1, \dots,
\alpha_n\}$ is $n$. Let $V(x_1,\dots, x_n, y_1, \dots, y_n)$ be
the algebraic locus of $(\alpha_1, \dots, \alpha_n,\E(\alpha_1), \dots,
\E(\alpha_n))$ over $K$. Then
$(\alpha_1, \dots, \alpha_n, \E(\alpha_1), \dots, \E(\alpha_n))$
is a generic point of $V$ over $K$. Moreover, $V$ is a
subvariety of $(\DG_a)^n\times (\DG_m)^n$.

Now, $V$ has some special algebro-geometric properties.
Firstly, the $x$ coordinates of a generic point are $\DQ$-linearly
independent. Secondly, any ``monomial'' relation
$\E(\alpha_1)^{m_1}\cdot\E(\alpha_2)^{m_2}\cdot \dots
\cdot\E(\alpha_n)^{m_n}=\beta$ (with $m_j\in \DZ$ for all $j$)
implies
\[
\sum m_j \alpha_j =\delta
\]
for some $\delta$ with $\E(\delta)=\beta$. ($\delta$ is defined
only up to translation by elements of $\Ker$.)

If there is in fact such a relation, we can reduce the study of
$(\alpha_1, \dots, \alpha_n)$ (and $V$) to a case of smaller $n$.
Thus it makes sense to assume about $\overline \alpha:=(\alpha_1,
\dots, \alpha_n)$ that there are no such relations.

Following Zilber, we call these assumptions on the $\bar x$ and $\bar y$ coordinates
of $V$, \emph{free from additive dependencies}
and \emph{free from multiplicative dependencies}, respectively. If $V$ satisfies both conditions we just say it is \emph{free}.

The Schanuel condition yields another constraint on (generic
points of) $V$. Assuming that $V$ is free, we easily deduce from
SC that the dimension of $V$ is at least $n$. But more
is true. Let $M$ be an $r\times n$ matrix over $\DZ$, of rank $r$.
Then $M{\overline\alpha}^T$ is a $\DQ$-linearly independent
$r$-tuple. Consider the values of $\E$ on the elements of the
$r$-tuple. These are monomials (depending only on $M$) in the
$\E(\alpha_j)$ (the $y_j$ in effect). Then SC implies that the
transcendence degree of
\[
M{\overline \alpha}^T \cup \{\text{the corresponding $\E$'s}\}
\]
is greater than or equal to $r$.

If $V$ has this property of generic points then we say it is
\emph{rotund}. (Zilber used the terms \emph{normal} and
\emph{ex-normal}.)

Thus in order to understand types in exponential fields satisfying
SC, one is inevitably led to varieties which are \emph{rotund
and free}.

We are finally able to define strongly exponentially-algebraically
closed.
\begin{definition}
An exponential field $F$ is a \emph{strongly exponentially-algebraically closed} if, given any rotund and free $V$ and any finitely generated subfield $K$ of $F$ over which $V$ is defined, there is a point in $V(F)$ of the form $(\overline \alpha, \overline{\E(\alpha)})$ which is generic in $V$ over $K$.
\end{definition}

\subsection{Extending automorphisms}\label{extending automorphisms}


The deepest model theory in Zilber's work has to do with
quasiminimal excellence. To understand this one has to go beyond
\cite{Z}, and the material is bound to be hard for those who
are not specialists in pure model theory. The main results of our
paper can be quickly obtained using quasiminimal excellence, but
we also indicate how they can also be obtained without it.

The key structural property of Zilber fields can be summarized
as follows:
\begin{prop}
Suppose $F$ is a Zilber field with CCP, and $F_0$ is a
finitely-generated partial E-subfield of $F$ which contains $SK$,
such that $F_0 \strong F$. Then any automorphism of $F_0$ extends to
an automorphism of $F$. In particular, the statement holds for any
countable Zilber field $F$.
\end{prop}
\begin{proof}[Sketch proof]
This follows from the quasiminimal excellence of the class of Zilber fields and Theorem~3.3 in \cite{K3}. Zilber uses a relational language whereas we use function symbols and the notion of partial exponential fields to give a simpler presentation. The notion of quasiminimal excellence depends critically on the language chosen, but one can translate from one language to the other to see that his proof does indeed work to prove our statement.
\end{proof}

We could avoid the use of excellence altogether. In the case where $F$ is countable and algebraic over $D(F) \cup \E(D(F))$, the proposition is a special case of part
of \cite[Theorem~7.2(2)]{K1}. This case is enough for our purposes.

\subsection{Countable subfields}\label{lowenheim}

\begin{lemma}
Let $F$ be any Zilber field, and let $X\subset F$ be a countable set.
Then there is a countable elementary subfield $F'\prec F$ such that $X\subset F'$ and $F'$
is also a Zilber field.
\end{lemma}

\begin{proof}
The result follows from Theorem~2 and section~5 of \cite{K4}. We sketch a simpler proof.

Without the requirement that $F'$ is a Zilber field, the result would follow immediately from the downward L\"owenheim-Skolem theorem for a countable theory. In order to obtain a Zilber field
we may need to add generic solutions to the free and rotund algebraic varieties. The idea is to construct a chain of structures, each an elementary substructure of $F$ and each of which contains the previous field and has algebraically generic realizations for the rotund and free varieties defined with parameters over the previous field. This
is a routine process. After constructing such chain of fields, one can define $F'$ to be the union, which will have all the necessary properties.
\end{proof}

\subsection{Proof of Theorem~2}\label{proof}

We begin by proving the theorem assuming we are in the countable
case and then use this to complete the general result. The
countable case will be proved by an automorphism argument.

First suppose that $F$ is a countable Zilber field (or more
generally, a Zilber field with CCP). We define an automorphism
$\sigma_1$ of $SK$ by defining $\sigma_1(2 \pi i) = -2\pi i$. Note
that this defines a unique automorphism, which restricts to
$\sigma_0$ on $\Qab$. Since $SK \strong F$,
Proposition~\ref{extending automorphisms} allows us to extend
$\sigma_1$ to an automorphism of $F$. Now if $\alpha \in \Qab
\smallsetminus \QabR$ then $\sigma_0(\alpha) \neq \alpha$, so
$\alpha$ is not pointwise definable in $K$.

Now let $\alpha \in \Qalg \smallsetminus \Qab$. Let $F_0 = SK(\alpha)$, with $D(F_0) = D(SK)$.
Then, since $\alpha$ is algebraic over $SK$ but not in $SK$, there is an automorphism $\sigma_2$
of $F_0$ which fixes $SK$ pointwise, but does not fix $\alpha$. Since $F_0$ is an algebraic
extension of $SK$ and the domain of exponentiation does not extend, the property $SK \strong F$
implies immediately that $F_0 \strong F$. Thus $\sigma_2$ extends to an automorphism of $F$, and
$\alpha$ is not pointwise definable in $F$.

If $F$ is an uncountable Zilber field and $\alpha\not\in \QabR$, Lemma~\ref{lowenheim} above
shows that there is a countable Zilber field $F'$ containing $\alpha$ and elementarily embedded in $F$.
We have shown that $\alpha$ is not definable in $F'$ which implies that $\alpha$ is not definable in $F$. That completes the proof of Theorem~2. \qed

\subsection{Orbits and definable points}
When $F$ is a Zilber field with CCP, we have shown that an algebraic number is in $\QabR$ if and only if its orbit under automorphisms of $F$ is a singleton.

\subsection{Other exponential fields}\label{otherfields}

The proof of Theorem~2 uses only that $SK$ admits the automorphism
$\sigma_1$, and that $F$ is built on top of it in such a
homogeneous way that the Proposition~\ref{extending automorphisms}
holds. For any non-zero algebraic number $\tau$, we can construct
a partial exponential field $CK_\tau$ which is like $SK$, but with
this $\tau$ as the generator of a cyclic kernel in place of the
usual transcendental generator. Then we can construct a strongly exponentially-algebraically
closed exponential field $\CB_\tau$, analogous to $\CB$ but with $CK_\tau$ in place of $SK$.
In this case there are two possibilities for what the definable algebraic numbers are.
Let $f$ be the minimal polynomial of $\tau$ over $\Qab$, and let $\bar{f}$ be the polynomial
obtained from it by applying the automorphism $\sigma_0$ of $\Qab$ to its coefficients.
If $\bar{f}(-\tau) = 0$ (for example, if $\tau = i$) then $\sigma_0$ extends to an involution
on the partial exponential field $CK_\tau$, and the definable algebraic numbers in $\CB_\tau$
are those in the fixed field of that involution. Otherwise (for example, if $\tau = 1$),
$CK_\tau$ has no non-trivial automorphisms, and the definable algebraic numbers are precisely
the elements of $CK_\tau$, that is, of $\Qab(\tau)$.

Similarly, one can build exponential fields on $SK$ (or on $CK_\tau$)
which are not strongly exponentially-algebraically closed, but
still have the required homogeneity properties for the proof of Theorem~2 to go
through, such as $SK^{EA}$, the free completion of $SK$ to an algebraically
closed exponential field, and $SK^{ELA}$, the free completion to an
algebraically closed exponential field with logarithms. See \cite{K1} for details of all these constructions.

\subsection{Extending involutions}

Although the involution $\sigma_1$ on $SK$ extends to some
automorphism of $\CB$, the extension is totally non-canonical, and
the question of whether it can be extended to an involution on
$\CB$ is open and appears to be very difficult.

\bibliographystyle{abbrv}
\bibliography{AlgebraicExponential}

\end{document}